\documentclass[a4paper,12pt]{amsart}
\usepackage{amssymb}
\usepackage{cite}
\usepackage{ifthen}
\usepackage[dvips]{graphicx}
\usepackage{booktabs} 
\usepackage{siunitx}  
\usepackage{amsmath, amssymb} 
\usepackage{subcaption} 
\usepackage[utf8]{inputenc}
\usepackage{graphicx}
\usepackage{booktabs}
\usepackage{caption}

\newcommand{\blue}{\color{blue}}
\nonstopmode \numberwithin{equation}{section}
\newtheorem*{theoC}{Theorem C}
\newtheorem*{theoD}{Theorem D}
\newtheorem*{theoE}{Theorem E}
\newtheorem*{theoF}{Theorem F}
\newtheorem*{theoG}{Theorem G}

\usepackage{amssymb}
\usepackage{ifthen}
\usepackage{graphicx}
\usepackage{amsmath}
\usepackage[T1]{fontenc} 
\usepackage[utf8]{inputenc}
\usepackage[usenames,dvipsnames]{color}
\usepackage{color}
\usepackage[english]{babel}
\usepackage{fancyhdr}
\usepackage{fancybox}
\usepackage{tikz}

\theoremstyle{plain}
\newtheorem{prop}{Proposition}

\newtheorem{conj}{Conjecture}

\theoremstyle{definition}
\newtheorem{defi}{Definition}[section]

\newtheorem{cor}{Corollary}[section]
\newtheorem{thm}{Theorem}[section]

\newtheorem{lem}{Lemma}[section]
\newtheorem{prob}{Problem}
\newtheorem{rem}{Remark}[section]

\theoremstyle{plain}

\newtheorem*{lemA}{Lemma A}
\newtheorem*{lemB}{Lemma B}

\newcounter{minutes}
\divide\time by 60
\newcounter{hours}
\multiply\time by 60
\addtocounter{minutes}{-\time}

\newcounter {own}
\def\theown {\thesection       .\arabic{own}}

\newenvironment{pf}[1][]{%
	\vskip 3mm
	\noindent
	\ifthenelse{\equal{#1}{}}%
	{{\slshape Proof. }}%
	{{\slshape #1.} }%
}%
{\qed\bigskip}

\newcounter{alphabet}

\def\be{\begin{equation}}
	\def\ee{\end{equation}}

\newcommand{\bee}{\begin{enumerate}}
	\newcommand{\eee}{\end{enumerate}}

\newcommand{\blem}{\begin{lem}}
	\newcommand{\elem}{\end{lem}}
\newcommand{\bthm}{\begin{thm}}
	\newcommand{\ethm}{\end{thm}}
\newcommand{\bcor}{\begin{cor}}
	\newcommand{\ecor}{\end{cor}}
\newcommand{\beg}{\begin{examp}}
	\newcommand{\eeg}{\end{examp}}
\newcommand{\begs}{\begin{examples}}
	\newcommand{\eegs}{\end{examples}}

\newcommand{\bdefn}{\begin{defn}}
	\newcommand{\edefn}{\end{defn}}

\newcommand{\bprob}{\begin{prob}}
	\newcommand{\eprob}{\end{prob}}
\newcommand{\bei}{\begin{itemize}}
	\newcommand{\eei}{\end{itemize}}

\newcommand{\bcon}{\begin{conj}}
	\newcommand{\econ}{\end{conj}}
\newcommand{\bcons}{\begin{conjs}}
	\newcommand{\econs}{\end{conjs}}
\newcommand{\bprop}{\begin{prop}}
	\newcommand{\eprop}{\end{prop}}
\newcommand{\br}{\begin{rem}}
	\newcommand{\er}{\end{rem}}
\newcommand{\brs}{\begin{rems}}
	\newcommand{\ers}{\end{rems}}
\newcommand{\bo}{\begin{obser}}
	\newcommand{\eo}{\end{obser}}
\newcommand{\bos}{\begin{obsers}}
	\newcommand{\eos}{\end{obsers}}
\newcommand{\bpf}{\begin{pf}}
	\newcommand{\epf}{\end{pf}}
\newcommand{\ba}{\begin{array}}
	\newcommand{\ea}{\end{array}}
\newcommand{\beq}{\begin{eqnarray}}
	\newcommand{\beqq}{\begin{eqnarray*}}
		\newcommand{\eeq}{\end{eqnarray}}
	\newcommand{\eeqq}{\end{eqnarray*}}

\begin{document}

\title{Sharp Landau-Type Theorems and Schlicht Disc Radii for certain Subclasses of Harmonic Mappings}

\author{Molla Basir Ahamed$^*$}
\address{Molla Basir Ahamed, Department of Mathematics, Jadavpur University, Kolkata-700032, West Bengal, India.}
\email{mbahamed.math@jadavpuruniversity.in}

\author{Rajesh Hossain}
\address{Rajesh Hossain, Department of Mathematics, Jadavpur University, Kolkata-700032, West Bengal, India.}
\email{rajesh1998hossain@gmail.com}

\def\thefootnote{}
\footnotetext{ {\tiny File:~\jobname.tex,
printed: \number\year-\number\month-\number\day,
          \thehours.\ifnum\theminutes<10{0}\fi\theminutes }
} \makeatletter\def\thefootnote{\@arabic\c@footnote}\makeatother
\begin{abstract} 
	Let $\mathcal{H}$ be the class of all complex-valued harmonic mappings $f=h+\overline{g}$ defined on the unit disc $\mathbb{D}=\{z\in\mathbb{C}:|z|<1\}$ with the normalization $h(0)=0=h'(0)-1$, here $h$ and $g$ are analytic functions in $\mathbb{D}$. In this paper, we investigates Landau-type theorems for several significant subclasses of sense-preserving harmonic mappings. Specifically, we establish sharp Landau-type theorems for the class $\mathcal{P}_{\mathcal{H}}^{0}(M)$ and the parameterized class $\mathcal{W}_{\mathcal{H}}^{0}(\alpha)$ for $\alpha \ge 0$. For mappings in $\mathcal{W}_{\mathcal{H}}^{0}(\alpha)$, we derive the radii of univalence and the radii of the largest schlicht discs contained in the images of the unit disc, expressing these results in terms of the Lerch Transcendent function $\Phi(z,s,a)$ and the Dilogarithm function ${\rm Li}_2(z)$. The sharpness of the obtained radii is demonstrated by constructing appropriate extremal functions for each class. These results generalize and extend various known Landau-type theorems in the theory of harmonic mappings.
\end{abstract}
\maketitle
\pagestyle{myheadings}
\markboth{M. B. Ahamed and R. Hossain}{Sharp Landau-Type Theorems for Several Classes of Harmonic Mappings}
\section{\bf Introduction}
A continuous complex-valued function $f=u+iv$ is a complex-valued harmonic function in a domain $\Omega\subset\mathbb{C}$, if both $u$ and $v$ are real-valued harmonic functions in $\Omega$ (see \cite{Duren-Har-2004}). The inverse function theorem and a result of Lewy \cite{Lewy-BAMS-1936} shows that a harmonic function $f$ is locally univalent (\textit{i.e.,} one-to-one) in $\Omega$ if, and only if, the Jacobian $J_f$ is non-zero in $\Omega$, where 
\begin{align*}
	J_f(z)=|f_z(z)|^2-|f_{\bar{z}}(z)|^2.
\end{align*}
A harmonic function $f$ is said to be sense-preserving if $ J_f(z)>0$ for $z\in\Omega$. Let $\mathcal{H}$ denote the class of complex-valued harmonic functions $f$ in the unit disc $\mathbb{D}=\{z\in\mathbb{D} : |z|<1\}$, normalized by $f(0)=0=f_z(0)-1$. Each function $f$ in $\mathcal{H}$ can be expressed as $f=h+\overline{g}$, where $h$ and $g$ are analytic functions in $\mathbb{D}$ (see \cite{Duren-Har-2004}). Here $h$ and $g$ are called the analytic and co-analytic part of $f$ respectively, and have the following power series representations 
\begin{align*}
	h(z)=z+\sum_{n=2}^{\infty}a_nz^n\;\;\mbox{and}\;\; g(z)=\sum_{n=1}^{\infty}b_nz^n.
\end{align*}
Let $\mathcal{S}_{\mathcal{H}}$ be the subclass of $\mathcal{H}$ consisting of univalent, and sense-preserving harmonic mappings on $\mathbb{D}$. Let $\mathcal{H}^0=\{f\in\mathcal{H}: f_{\bar{z}}(0)=0\}$, and $\mathcal{S}^0_{\mathcal{H}}=\{f\in\mathcal{S}_{\mathcal{H}}: f_{\bar{z}}(0)=0\}$. Hence for any function $f=h+\overline{g}$ in $\mathcal{H}^0$, its analytic and co-analytic parts can be represented by 
\begin{align}\label{Eq-1.1B}
	h(z)=z+\sum_{n=2}^{\infty}a_nz^n\;\;\mbox{and}\;\; g(z)=\sum_{n=2}^{\infty}b_nz^n,
\end{align}
respectively. It is interesting to note that $\mathcal{S}_{\mathcal{H}}$ reduced to $\mathcal{S}$, the class of normalized analytic and univalent functions in $\mathbb{D}$, if the co-analytic part of functions in the class $\mathcal{S}_{\mathcal{H}}$ is zero. In $1984$, Clunie and Shell-Small \cite{Clunie-Sheil-Small-AASF-1984} investigated the class $\mathcal{S}_{\mathcal{H}}$, together with some geometric subclasses. Subsequently, the class $\mathcal{S}_{\mathcal{H}}$ and its subclasses have been extensively studied by several authors (see \cite{Bshouty-Lyzzaik-2011,Bshouty-Joshi-Joshi-2013,Clunie-Sheil-Small-AASF-1984,Kalaj-Punn-Vuorinen-2014,Wang-Liang-2001}.)\vspace{1.2mm}

For a continuously differentiable function $f$, we denote $\Lambda_f$ and $\lambda_f$ by
\begin{align*}
	\Lambda_f=\max_{0\leq\theta\leq 2\pi}|f_z+e^{-2i\theta}f_{\bar{z}}|=|f_z|+|f_{\bar{z}}|
\end{align*}
and 
\begin{align*}
	\lambda_f=\min_{0\leq\theta\leq 2\pi}|f_z+e^{-2i\theta}f_{\bar{z}}|=||f_z|-|f_{\bar{z}}||.
\end{align*}
Thus, it is easy to see that for sense-preserving harmonic mapping $f$, one has $J_f=\Lambda_f\lambda_f$. A harmonic mapping $f=h+\overline{g}$ defined on the unit disc $\mathbb{D}$ can be expressed as 
\begin{align*}
	f\left(re^{i\theta}\right)=\sum_{n=0}^{\infty}a_nr^ne^{in\theta}+\sum_{n=1}^{\infty}\overline{b_n}r^ne^{-in\theta},\; 0\leq r<1,
\end{align*}
where 
\begin{align*}
	h(z)=\sum_{n=0}^{\infty}a_nz^n\;\; \mbox{and}\;\; g(z)=\sum_{n=1}^{\infty}b_nz^n.
\end{align*}
The classical Landau theorem (see \cite{Landau-192}) asserts that if $f$ is a holomorphic mapping with $f(0)=0=f'(0)-1$ and $|f(z)|<M$ for $z\in\mathbb{D}$, then $f$ is univalent in $\mathbb{D}_{r_0}$, and $f(\mathbb{D}_{r_0})$ contains a disc $\mathbb{D}_{\sigma_0}$, where
\begin{align*}
	r_0=\frac{1}{M+\sqrt{M^2-1}}\;\;\mbox{and}\; \sigma_0=Mr_0^2.
\end{align*}
The quantities $r_0$ and $\sigma_0$ cannot be improved, the extremal function is 
\begin{align*}
	f_0(z)=Mz\left(\frac{1-Mz}{M-z}\right).
\end{align*}
For a comprehensive look at recent advancements in Landau-type theorems for various function classes, including harmonic, biharmonic, and polyanalytic mappings, we refer to the works \cite{Allu-Kumar-JMAA-2024,Chen-Punno-Rasila-2014,Lui-CMA-2009,Liu-Xu-MM-2024,Wang-Zhong-CMFT-2025,Li-Liu-Ponnusamy-Zhao-JMAA-2026,Abdulhadi-Abu-Muhanna-2006,Abdulhadi-Hajj-2022,Chen-Ponnusamy-Wang-2011,Chen-Ponnusamy-Wang-2010,Chen-Zhu-2019,Liu-Chen-2018,Liu-Luo-Luo-2020} and the references therein. These studies extend the classical theory by establishing sharp univalence radii and schlicht disk dimensions for specialized subclasses such as sense-preserving harmonic mappings and log-$\alpha$-analytic functions.\vspace{1.2mm}

However, for holomorphic functions $f$ on the unit disc $\mathbb{D}$ with $f'(0)=1$, there is a Bloch theorem (see \cite{Bloch-Les-1925,Huang-JMAA-2008}) which asserts the existence of a positive constant $b$ such that $f(\mathbb{D})$ contains a schlicht disc of radius $b$. A disc $D$ is said to be schlicht disc if there exists a region $\Omega$ in the unit disc $\mathbb{D}$ such that $f$ is univalent on $\Omega$ and $f(\Omega)=D$. Let $b(f)$ denote the least upper bound of the radii of all schlicht discs that $f$ carries and $\mathcal{F}$ denote the set of all holomorphic functions defined on $\overline{\mathbb{D}}:=\{z :|z|\leq 1\}$ satisfying $|f'(0)|=1$, then the Bloch constant is the number defined by 
\begin{align*}
	\beta(\mathcal{F})=\inf\{b(f) : f\in\mathcal{F}\}.
\end{align*}
If one considers the function $f(z)=z$, then clearly $\beta(\mathcal{F})\leq 1$. While the exact value of $\beta(\mathcal{F})$ is still unknown, better estimates have been established. In $1996$, Chen and Gauthier \cite{Chen-Gauthier-1996} proved that $\beta(\mathcal{F})$ lies within the interval $[0.4330, 0.4719]$, \textit{i.e.}, $0.4330\leq \beta(\mathcal{F})\leq 0.4719$.\vspace{1.2mm}

In $2000$, Chen \emph{et al.} \cite{Chen-Gauthier-Hengartner-2000} established the following two versions of Landau-Bloch type theorems for bounded harmonic mapping on the disc under a suitable restriction.
Theorems A and B are not sharp. In $2006$, better estimates for Theorems A and B were given by Grigoryan \cite{Grigoyan-CV-2006}. Specifically, Grigoryan proved the following lemma.
\begin{lemA}\cite{Grigoyan-CV-2006}\label{Lem-AA}
Assume that $f=h+\overline{g}$ with $h$ and $g$ analytic in the unit disc $\mathbb{D}$ and $h(z)=\sum_{n=1}^{\infty}a_nz^n$ and $g(z)=\sum_{n=1}^{\infty}b_nz^n$ for $z\in\mathbb{D}$.
\begin{enumerate}
	\item[(a)] If $|f(z)|<M$ for $z\in\mathbb{D}$, then 
	\begin{align*}
		|a_n|, \; |b_n|\leq M,\; n=1, 2, \ldots
	\end{align*} 
	\item[(b)] If $\Lambda_f\leq \Lambda$ for $z\in\mathbb{D}$, then 
	\begin{align*}
		|a_n|+|b_n|\leq \frac{\Lambda}{n},\; n=1, 2, \ldots
	\end{align*}
\end{enumerate}
\end{lemA}
\begin{lemB}(The Schwarz Lemma, \cite{Duren-Har-2004}).
	Let $f$ be a harmonic mapping of the unit disk $\mathbb{D}$
	such that $f(0) = 0$ and $f(\mathbb{D}) \subset \mathbb{D}$. Then
	\begin{align}\label{Eq-1.1}
		\Lambda_f(0)\leq\frac{4}{\pi},
	\end{align}
	\begin{align}
			\Lambda_f(z)\leq\frac{8}{\pi(1-|z|^2)},\;\mbox{for}\;z\in\mathbb{D},
	\end{align}
	\begin{align}
	        |f(z)|\leq\frac{4}{\pi}\arctan|z|\leq\frac{4}{\pi}|z|,\;\mbox{for}\;z\in\mathbb{D}.
	\end{align}
\end{lemB}
The following result has been proved by Grigoryan using Lemma A, which improved the estimates in Theorems A and B.
\begin{theoC}\cite{Grigoyan-CV-2006}
	Let $f$ be harmonic mapping of the unit disc $\mathbb{D}$ such that $f(0)=0$, $J_f(0)=1$ and $|f(z)|<M$ for $z\in\mathbb{D}$. Then, $f$ is univalent on a disc $\mathbb{D}_{\rho_1}$ with 
	\begin{align*}
		\rho_1=1-\frac{2\sqrt{2}M}{\sqrt{\pi+8M^2}}
	\end{align*}
	and $f(\mathbb{D}_{\rho_1})$ contains a schlicht disc $\mathbb{D}_{R_1}$ with 
	\begin{align*}
		R_1=\frac{\pi}{4M}+4M-4M\sqrt{1+\frac{\pi}{8M^2}}.
	\end{align*}
\end{theoC}
\begin{theoD}\cite{Grigoyan-CV-2006}
	Let $f$ be harmonic mapping of the unit disc $\mathbb{D}$ such that $f(0)=0$, $\lambda_f(0)=1$ and $\Lambda_f(z)<\Lambda$ for all $z\in\mathbb{D}$. Then, $f$ is univalent on a disc $\mathbb{D}_{\rho_2}$ with 
	\begin{align*}
		\rho_2=\rho_2(\Lambda)=\frac{1}{1+\Lambda}
	\end{align*}
	and $f(\mathbb{D}_{\rho_2})$ contains a schlicht disc $\mathbb{D}_{R_2}$ with 
	\begin{align*}
		R_2(\Lambda)=1-\Lambda\ln\left(1+\frac{1}{\Lambda}\right).
	\end{align*}
\end{theoD}
The following harmonic analogue of Noshiro-Warchawsk theorem was proved by Mocanu in $1981$.
\begin{theoE}\cite{Partyka-Sakan-AASFM-2007}
	Suppose $f=h+\overline{g}$ is a harmonic function in a convex domain $D$ such that for some real number, $\gamma$ ${\rm Re}\left(e^{i\gamma}h'(z)\right)>|g'(z)|$ for all $z\in D$. Then $f$ is univalent in $D$.
\end{theoE}
\subsection{\bf Certain classes of harmonic mappings:} Ponnusamy \emph{et al.} \cite{Ponnusamy-Yamamoto-Yanagihara-CV-2013} introduced the following subclass of $\mathcal{H}$ as follows
\begin{align*}
	\mathcal{C}^1_{\mathcal{H}}=\{f=h+\overline{g} : {\rm Re}\left(e^{i\gamma}h'(z)\right)>|g'(z)|\; \mbox{for}\; z\in\mathbb{D}\}.
\end{align*}
In view of Theorem E, every function in $\mathcal{C}^1_{\mathcal{H}}$ is univalent. Moreover, Ponnusamy \emph{et al.} \cite{Ponnusamy-Yamamoto-Yanagihara-CV-2013} proved that functions in $\mathcal{C}^1_{\mathcal{H}}$ are also close-to-convex harmonic functions. Subsequently, Li and Ponnusamy \cite{Li-Ponnusamy-NA-2013} have introduced a more general subclass of harmonic functions which contains $\mathcal{C}^1_{\mathcal{H}}$ as a subclass, defining it as
\begin{align*}
	\mathcal{P}^0_{\mathcal{H}}(\alpha):=\{f=h+\overline{g} : {\rm Re}\left(h'(z)-\alpha\right)>|g'(z)|\; \mbox{for}\; z\in\mathbb{D}\}.
\end{align*}
For $0\leq\alpha<1$, Li and Ponnusamy \cite{Li-Ponnusamy-NA-2013} proved that functions in $\mathcal{P}^0_{\mathcal{H}}(\alpha)$ are univalent. Furthermore, they also studied coefficient bounds and sections of univalence for functions in this class.\vspace{1.2mm}

Nagpal and Ravichandran \cite{Nagpal-Ravichandran-JKMS-2014} have defined a new subclass $\mathcal{W}^0_{\mathcal{H}}$ of univalent harmonic functions on $\mathbb{D}$ by
\begin{align*}
	\mathcal{W}^0_{\mathcal{H}}=\{f=h+\overline{g}\in\mathcal{H} : {\rm Re}\left(h'(z)+zh''(z)\right)>|g'(z)+zg''(z)|\; \mbox{for}\; z\in\mathbb{D}\}.
\end{align*} 
Bshouty and Lyzzaik \cite{Bshouty-Lyzzaik-2011}, and later Bshouty \emph{et al.} \cite{Bshouty-Joshi-Joshi-2013}, constructed new harmonic subclasses by imposing restrictions on the analytic part of $f$. Motivated by the above, Allu and Ghosh \cite{Ghosh-Vasudevarao-CV-2019} have introduced a new subclass $\mathcal{W}^0_{\mathcal{H}}(\alpha)$ as 
\begin{align*}
	\mathcal{W}^0_{\mathcal{H}}(\alpha):=\{f=h+\overline{g}\in\mathcal{H} : {\rm Re}\left(h'(z)+\alpha zh''(z)\right)>|g'(z)+\alpha zg''(z)|\; \mbox{for}\; z\in\mathbb{D}\}
\end{align*}
with $g'(0)=0$. \vspace{1.2mm}

For $M>0$, the following subclasses are defined as follows:
\begin{align*}
	\mathcal{B}^0_{\mathcal{H}}(M):&=\{f=h+\overline{g}\in\mathcal{H}^0 : |zh''(z)|\leq M-|zg''(z)|\; \mbox{for}\; z\in\mathbb{D}\},\\
	\mathcal{P}^0_{\mathcal{H}}(M):&=\{f=h+\overline{g}\in\mathcal{H}^0 : {\rm Re}\left(zh''(z)\right)> M-|zg''(z)|\; \mbox{for}\; z\in\mathbb{D}\},
\end{align*}
where $\mathcal{H}^0:=\{f\in\mathcal{H} : f_{\bar{z}}(0)=0\}$.\vspace{1.2mm}

It has been shown in \cite{Ghosh-Vasudevarao-CV-2018} that the harmonic subclasses $\mathcal{B}^0_{\mathcal{H}}(M)$ and $\mathcal{P}^0_{\mathcal{H}}(M)$ are closely related to the analytic subclasses $\mathcal{B}(M)$ and $\mathcal{P}(M)$, respectively. The analytic subclasses are defined by
\begin{align*}
	\mathcal{B}(M):&=\{\phi\in\mathcal{A}: |z\phi''(z)|\leq M\; \mbox{for}\; z\in\mathbb{D}\},\\
	\mathcal{P}(M):&=\{\phi\in\mathcal{A} : {\rm Re}\left(zh''(z)\right)> M\; \mbox{for}\; z\in\mathbb{D}\}.
\end{align*}
The subclasses $\mathcal{B}(M)$ and $\mathcal{P}(M)$ have been extensively studied by Mocanu \cite{Mocanu-1992}, Ponnusamy \emph{et al.} \cite{Ponnusamy-Singh-BBMS-2002}. Ponnusamy \emph{et al.} \cite{Ponnusamy-Singh-BBMS-2002} showed that functions in $\mathcal{B}(M)$ are univalent and starlike for $0<M\leq 1$, and convex for $0<M\leq 1/2$. Furthermore, the region of variability for functions in $\mathcal{B}(M)$ was studied by Ponnusamy \emph{et al.} \cite{Ponnusamy-Vasudevarao-Vuorinen-CV-2009}. In $1995$, Ali \emph{et al.} \cite{Ali-Ponnusamy-Singh-APM-1995} proved that each function in the class $\mathcal{P}(M)$ is univalent and starlike for $0<M<1/\log 4$. Therefore, finding analytic and geometric characterizations for functions in the related harmonic classes $\mathcal{B}^0_{\mathcal{H}}(M)$ and $\mathcal{P}^0_{\mathcal{H}}(M)$ is a natural extension in geometric function theory.\vspace{1.2mm}

{The central objective of the present study is to investigate the geometric properties of several subclasses of sense-preserving harmonic mappings, with a particular focus on establishing sharp Landau-type theorems. While the classical theory for analytic functions provides a well-defined framework for Landau-Bloch constants, the transition to the harmonic setting—characterized by the intricate interplay between the analytic part $h$ and the co-analytic part $g$—presents significant analytical challenges. To address these, we introduce a unified approach to determine both the radii of univalence and the radii of the largest schlicht discs for the class $\mathcal{P}_{\mathcal{H}}^{0}(M)$, the parameterized family $\mathcal{W}_{\mathcal{H}}^{0}(\alpha)$, and the class $\mathcal{G}_{k,H}(\alpha, 1)$ defined via second-order differential inequalities. By employing the Lerch Transcendent $\Phi(z, s, a)$ and the Dilogarithm function ${\rm Li}_2(z)$, we derive explicit quantitative bounds that offer a refined characterization of the local univalence of these mappings. Beyond establishing new results, a key aim of this work is to demonstrate how these generalized estimates reduce to classical theorems in the boundary cases (notably $\alpha=0$ and $\alpha=1$), thereby unifying several disparate results in the literature. Finally, we establish the sharpness of each obtained radius by constructing specific extremal functions, ensuring that the constants provided are the best possible. Through this systematic investigation, we aim to contribute new analytical tools and sharp estimates to the broader study of harmonic geometric function theory.}\vspace{1.2mm}

{The primary novelty of this work lies in the sophisticated analytical framework employed to establish sharp Landau-type theorems for the class $\mathcal{W}_{\mathcal{H}}^{0}(\alpha)$. By utilizing the Lerch Transcendent $\Phi(z,s,a)$, we provide a unified approach to evaluating the infinite series of coefficients that arise from the second-order differential inequalities defining this class. This methodology allows for a precise characterization of the univalence radii across a continuous range of the parameter $\alpha$, bridging the gap between several previously disparate results in the literature. Furthermore, this study contributes significantly to geometric function theory by extending well-known analytic results to the harmonic setting. While Landau-type constants for analytic functions are extensively studied, the transition to harmonic mappings $f=h+\overline{g}$ introduces substantial complexity due to the interaction between the analytic and co-analytic parts. Our derivation of sharp constants for mappings satisfying specific second-order differential conditions offers new insights into the growth and covering properties of harmonic functions, effectively addressing a non-trivial problem in modern complex analysis.}
	\section{\bf Main results}
	{To establish the sharp radii for the harmonic subclasses investigated in this section, we utilize a unified analytical framework based on specialized transcendental functions. Central to our derivation for the class $\mathcal{W}_{\mathcal{H}}^{0}(\alpha)$ is the Lerch Transcendent $\Phi(z,s,a) $ Specifically, we employ the case where $s=1$, which facilitates the evaluation of infinite series arising from second-order differential inequalities and allows for the precise determination of univalence radii as roots of equations involving $\Phi(r, 1, 1/\alpha)$. Additionally, for the boundary case $\alpha=1$, the covering radius is expressed using the Dilogarithm function ${\rm Li}_2(z)$, which serves as an integral representation of the Lerch Transcendent in the evaluation of analytic parts. By integrating these special functions directly into our proofs, we provide a more refined and systematic characterization of the local univalence and covering properties than previously available in the literature. 
	}
	\begin{defi}(Lerch Transcendent). The Lerch Transcendent function $\Phi(z, s, a)$ is defined by the following power series $$\Phi(z, s, a) = \sum_{k=0}^{\infty} \frac{z^k}{(k + a)^s}$$This series converges for $|z| < 1$, and for $|z|= 1$ provided that $\text{Re}(s) > 1$. In our study, we utilize the specific case where $s=1$, which allows us to evaluate the infinite series resulting from second-order differential inequalities.
	\end{defi}
	\begin{defi} The Dilogarithm function ${\rm Li}_2(z)$, which appears in the covering radius for the boundary case $\alpha=1$, is defined as $${\rm Li}_2(z) = \sum_{n=1}^{\infty} \frac{z^n}{n^2}.$$
	\end{defi}
	There is a significant relationship between the Lerch Transcendent and other special functions used in this paper:
	When $s=1$ and $a=1$, the function relates to the natural logarithm as 
	\begin{align*}
		\Phi(z, 1, 1) = -\frac{\ln(1-z)}{z}.
	\end{align*}
	In the specific case of Corollary 3.2 where $\alpha = 1$, the integration of the analytic parts involves terms that can be represented by ${\rm Li}_2(z)$, which can be seen as an integral form of the Lerch Transcendent9. Specifically, for functions in $\mathcal{W}_{\mathcal{H}}^{0}(1)$, the radius of the schlicht disc is expressed as 
	\begin{align*}
		R_2 = \left(\frac{\pi}{4M} + 2\right)\rho_2 - 2{\rm Li}_2(\rho_2).
	\end{align*}
	These functions enable the derivation of the sharp univalence radius $\rho_2$ as the root of the equation involving the series 
	\begin{align*}
		\sum_{n=2}^{\infty} \frac{1}{\alpha n + (1-\alpha)}r^{n-1}=\Phi\left(r, 1, \frac{1}{\alpha}\right).
	\end{align*}
	\subsection{\bf Landau theorem for harmonic mappings in the class $ \mathcal{P}^0_{\mathcal{H}}(M)$}
In $2019$, Ghosh and Allu \cite{Ghosh-Vasudevarao-CV-2019}, proved the following result finding the coefficient estimates of functions in the class $\mathcal{P}^0_{\mathcal{H}}(M)$.
\begin{theoG}\cite{Ghosh-Vasudevarao-CV-2019}
	Let $f=h+\overline{g}\in \mathcal{P}^0_{\mathcal{H}}(M)$ for $\alpha\geq 0$ and be of the form $(1.2)$. then for $n\geq 2$
	\begin{enumerate}
		\item[(i)] $|a_n|+|b_n|\leq\dfrac{2M}{n(n-1)};$ \vspace{2mm}
		\item[(ii)] $||a_n|-|b_n||\leq\dfrac{2M}{n(n-1)};$\vspace{2mm}
		\item[(iii)] $|a_n|\leq\dfrac{2M}{n(n-1)}.$  
	\end{enumerate}
	All these results are sharp for the function $f_{M}$ given by 
	\begin{align*}
		f'_{M}(z)=1-2M\ln (1-z).
	\end{align*}
\end{theoG}
In Theorem \ref{Th-2.1}, we establish a sharp Landau-type result for the class $\mathcal{P}_{\mathcal{H}}^{0}(M)$, proving that any harmonic mapping in this class is univalent within a specific disc $D_{\rho_1}$. We also determine that the image $f(D_{\rho_1})$ contains a schlicht disc of a precise radius $R_1$, where both $\rho_1$ and $R_1$ are expressed in terms of the constant $M$. To demonstrate the sharpness of these findings, we construct an extremal function that proves both the univalence and covering radii are best possible. This result generalizes several existing theorems in harmonic mapping theory by providing sharp constants for this specialized subclass.
\begin{thm}\label{Th-2.1}
	Let $f\in\mathcal{P}_H^{0}(M)$ be a harmonic mapping on the unit disc $\mathbb{D}$ such that $f(0)=0, J_f(0)=1$ and $|f(z)|<M$ for $z\in\mathbb{D}$. Then, $f$ is univalent on a disc $\mathbb{D}_{\rho_1}$ with
	\begin{align*}
		\rho_1=1-e^{-\frac{\pi}{8M^2}},
	\end{align*}
	and $f(\mathbb{D}_{\rho_1})$ contains a schlicht disc with
	\begin{align*}
		R_1=\frac{\pi}{4M}+2M-\left(\frac{\pi}{2M}+2M\right)e^{-\frac{\pi}{8M^2}}.
	\end{align*}
	This result is sharp, with an extremal function given by
	\begin{align}\label{Eq-2.1A}
		F_1(z)=\frac{\pi}{4M}z+2M\left(z+(1-z)\log(1-z)\right).
	\end{align}
\end{thm}
\begin{rem}
	The numerical values of the sharp radii $\rho_1$ and $R_1$ for representative values of $M$ are summarized in Table 1, with their corresponding geometric representations illustrated in Figure 1.
\end{rem}
\begin{table}[htbp]
	\centering
	\label{Tab-1}
	\begin{tabular}{@{} S[table-format=1.1] S[table-format=1.4] S[table-format=1.4] @{}}
		\toprule
		{Constant ($M$)} & {Univalence Radius ($\rho_1$)} & {Schlicht Radius ($R_1$)} \\
		\midrule
		1.0 & 0.3247 & 0.1706 \\
		1.5 & 0.1603 & 0.0683 \\
		2.0 & 0.0934 & 0.0381 \\
		\bottomrule
	\end{tabular}
	\caption{Sharp radii for the class $\mathcal{P}_{\mathcal{H}}^{0}(M)$ for various values of $M=1.0$, $1.5$ and $2.0$.}
\end{table}
\begin{figure}[htbp]
	\centering
	\includegraphics[width=0.7\textwidth]{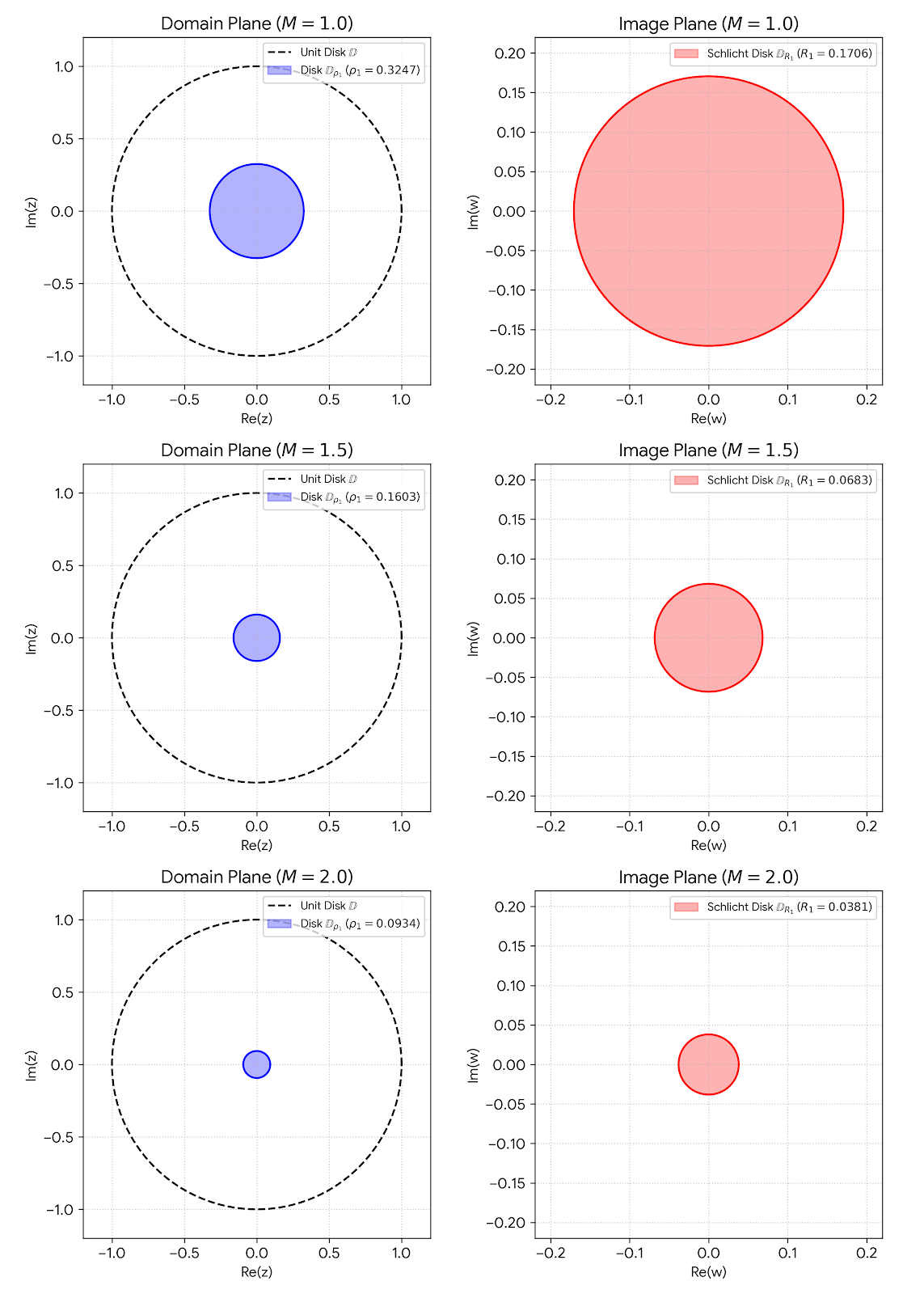} 
	\caption{Sharp geometric radii for mappings in $\mathcal{P}_{\mathcal{H}}^{0}(M)$. The left panels show the univalence disks $\mathbb{D}_{\rho_1}$ in the unit disk $\mathbb{D}$, and the right panels show the guaranteed schlicht disks $\mathbb{D}_{R_1}$. The plots illustrate the inverse relationship between the boundedness constant $M$ and the radii $\rho_1$ and $R_1$, emphasizing the sharp nature of the constants derived in Theorem \ref{Th-2.1}}
	\label{fig:landau_comparison}
\end{figure}
\begin{proof}[\bf Proof of Theorem \ref{Th-2.1}]
	Fix $r\in(0,1)$ and $z_1\neq z_2$. A simple computation leads to
	\begin{align}\label{Eq-2.2}
		f(z_1)-f(z_2)=\int_{[z_1,z_2]}f_z(z)dz+f_{\overline{z}}(z)d\overline{z}=\int_{[z_1,z_2]}h'(z)dz+{\overline{g'}}(z)d\overline{z},
	\end{align}
	where $\gamma=[z_1,z_2], \gamma(t)=z_1(1-t)+tz_2$. By the Schwarz lemma, we have
	\begin{align}\label{Eq-22.33}
		\lambda_f(0)=\frac{1}{\Lambda_f(0)}\geq\frac{\pi}{4M}.
	\end{align} 
	Combining \eqref{Eq-2.2} and \eqref{Eq-22.33} via the triangle inequality, we obtain
	\begin{align*}
		|f(z_1)-f(z_2)|&\geq\bigg|\int_{[z_1,z_2]}h'(z)dz+{\overline{g'}}(z)d\overline{z}\bigg|\\&\geq\bigg|\int_{[z_1,z_2]}h'(0)dz+{\overline{g'}}(0)d\overline{z}\bigg|\\&-\bigg|\int_{[z_1,z_2]}(h'(z)-h'(0))dz+({\overline{g'}}(z)-{\overline{g'}}(0))d\overline{z}\bigg|\\&\geq\lambda_f(0)|z_1-z_2|-|h(z_2)-h(z_1)-h'(0)(z_1-z_2)|\\&\quad-|g(z_2)-g(z_1)-g'(0)(z_1-z_2)|\\&\geq|z_1-z_2|\left(\frac{\pi}{4M}-\sum_{n=2}^{\infty}(|a_n|+|b_n|)nr^{n-1}\right).
	\end{align*}
	Then by Lemma A, using the coefficient estimate, we obtain
	\begin{align}\label{Eq-2.1}
		|f(z_1)-f(z_2)|&\geq|z_1-z_2|\left(\frac{\pi}{4M}-\sum_{n=2}^{\infty}\frac{2M}{n(n-1)}nr^{n-1}\right)\nonumber\\&\geq|z_1-z_2|\left(\frac{\pi}{4M}+2M\log(1-r)\right).
	\end{align}
	Let 
	\begin{align*}
		\mathcal{J}_1(r)=\frac{\pi}{4M}+2M\log(1-r),
	\end{align*}
	we have to check monotonocity of $\eta(r)$. Therefore, 
	\begin{align*}
		\mathcal{J}_1(0)=\frac{\pi}{4M}{\blue >0}\;\mbox{and}\;\lim_{r\rightarrow1^{-}}\mathcal{J}_1(r)=-\infty<0.
	\end{align*}
	\begin{align*}
		\mathcal{J}_1'(r) = -\frac{2M}{1-r} < 0 \quad \text{for } r \in (0, 1).
	\end{align*}
	This implies that $\mathcal{J}_1(r)$ is a strictly decreasing function on $(0, 1)$. Since $\mathcal{J}_1$ is continuous, the Intermediate Value Theorem ensures the existence of a root, say $\rho_1 \in (0, 1)$.
	 From equation \eqref{Eq-2.1}, it follows that $|f(z_1) - f(z_2)| > 0$ for all $|z_1|, |z_2| < \rho_1$ with $z_1 \neq z_2$. This establishes the univalence of $f$ in the disc $\mathbb{D}_{\rho_1}$.\vspace{1.2mm}
	
	Let $|z|=\rho_1$. Then, we have
	\begin{align*}
		|f(z)|&\geq|a_1z+b_1\overline{z}|-\sum_{n=2}^{\infty}(|a_n|+|b_n|)|z|^{n}\\&\geq\frac{\pi}{4M}\rho_1+2M\left(\rho_1+(1-\rho_1)\log(1-\rho_1)\right):=R_1.
	\end{align*}
	To show that the univalent radius $\rho_1$ is sharp, we need to prove that $F_1(z)$	is not univalent in $\mathbb{D}_r$ for each $r \in (\rho_1, 1]$. In fact, considering the real differentiable
	function
	\begin{align}\label{Eq-2.3}
		h_0(x) =\frac{\pi}{4M}x+2M\left(x+(1-x)\log(1-x)\right), \quad x \in [0, 1].
	\end{align}
	Since the continuous function $h_0'(x) = \frac{\pi}{4M} - 2M\log(1-x)$ is strictly decreasing on $[0, 1]$ and $h_0'(\rho_1) = 0$, it follows that $h_0'(x) > 0$ for $x \in [0, \rho_1)$ and $h_0'(x) < 0$ for $x \in (\rho_1, 1]$. Consequently, $h_0(x)$ is strictly increasing on $[0, \rho_1]$ and strictly decreasing on $[\rho_1, 1]$. Since $h_0(0) = 0$, there is a unique real $r_1\in (\rho_1, 1]$
	such that $h_0(r_1) = 0$ if $\lim_{x\rightarrow1^{-}}h_0(x) \leq 0$, and
	\begin{align}\label{Eq-2.4}
		R_1 =\frac{\pi}{4M}\rho_1+2M\left(\rho_1+(1-\rho_1)\log(1-\rho_1)\right)=h_0(\rho_1) > h_0(0) = 0.
	\end{align}
	For every fixed $r \in (\rho_1, 1]$, set $x_1 = \rho_1 + \varepsilon$, where
	\[ \varepsilon = \begin{cases}
		\min\left\{ \dfrac{r-\rho_1}{2}, \dfrac{r_1-\rho_1}{2} \right\}, & \text{if } f_0(1) \leq 0,\vspace{2mm} \\
		\dfrac{r-\rho_1}{2}, & \text{if } f_0(1) > 0.
	\end{cases} \]
	by the mean value theorem, there is a unique $\delta \in (0, \rho_1)$ such that $x_2 :=
	\rho_1 - \delta \in (0, \rho_1)$ and $h_0(x_1) = h_0(x_2)$.
	Let $z_1 = x_1$ and $z_2 = x_2$. Then $z_1, z_2 \in\mathbb{D}_r$ with $z_1 \neq z_2$. A simple
	computation leads to
	\[F_1(z_1) = F_1(x_1) = h_0(x_1) = h_0(x_2) = F_1(z_2).\]
	Hence $F_1$ is not univalent in the disk $\mathbb{D}_r$ for each $r \in (\rho_1, 1]$, and the univalent
	radius $\rho_1$ is sharp.\vspace*{2mm}

	Finally, note that $F_1(0) = 0$ and picking up $z' = \rho_1\in \partial \mathbb{D}_{\rho_1}$, by \eqref{Eq-2.1A},
	\eqref{Eq-2.3} and \eqref{Eq-2.4}, we have
	\begin{align*}
		|F_1(z') - F_1(0)| = |F_1(\rho_1)| = |h_0(\rho_1)| = h_0(\rho_1) = R_1.
	\end{align*}
	Hence, the covering radius $R_1$ is also sharp.
	This completes the proof.
\end{proof}
\subsection{\bf Landau theorem for harmonic mappings in the class $ \mathcal{W}_H^0(\alpha)$}
For $\alpha\geq 0$, each function in $\mathcal{W}(\alpha)$ is close-to-convex, and hence in view of, \cite[Lemma 2.1]{Ghosh-Vasudevarao-CV-2018}, $\mathcal{W}^0_{\mathcal{H}}(\alpha)$ is a subclass of $\mathcal{C}^0_{\mathcal{H}}$. Chichra \cite{Chichra-PAMS-1977} has shown that if $0\leq\beta\leq\alpha$, then $\mathcal{W}(\alpha)\subset\mathcal{W}(\beta)$, hence $\mathcal{W}^0_{\mathcal{H}}(\alpha)\subset \mathcal{W}^0_{\mathcal{H}}(\beta)$. Therefore, $\mathcal{W}(\alpha)$ is starlike for $\alpha\geq 1$ since $\mathcal{W}(1)$ is starlike. Also, if $f\in \mathcal{W}^0_{\mathcal{H}}(\alpha)$ and $\alpha\geq 1$, then $h+\epsilon g\in\mathcal{W}(\alpha)$ is starlike function in $\mathbb{D}$ for each $\epsilon\; (|\epsilon|=1)$. Thus by \cite[Theorem 1.8.2 ]{Nagpal-Ravichandran-JKMS-2014}  we see that $\mathcal{W}^0_{\mathcal{H}}(\alpha)\subset \mathcal{S}^0_{\mathcal{H}}$. In particular, when $\alpha\geq 1$, functions in $\mathcal{W}^0_{\mathcal{H}}$ are fully starlike.\vspace{1.2mm}

In $2018$, Ghosh and Allu \cite{Ghosh-Vasudevarao-CV-2018}, established the following coefficient estimates for functions in the class $\mathcal{W}^0_{\mathcal{H}}(\alpha)$.
\begin{theoF}\cite{Ghosh-Vasudevarao-CV-2019}
	Let $f=h+\overline{g}\in \mathcal{W}^0_{\mathcal{H}}(\alpha)$ for $\alpha\geq 0$ and be of the form \eqref{Eq-1.1B}. Then for $n\geq 2$
	\begin{enumerate}
		\item[(i)] $|a_n|+|b_n|\leq\dfrac{2}{\alpha n^2+n(1-\alpha)};$ \vspace{2mm}
		\item[(ii)] $||a_n|-|b_n||\leq\dfrac{2}{\alpha n^2+n(1-\alpha)};$\vspace{2mm}
		\item[(iii)] $|a_n|\leq\dfrac{2}{\alpha n^2+n(1-\alpha)}.$  
	\end{enumerate}
	All these results are sharp for the function $f_{\alpha}$ given by 
	\begin{align*}
		f_{\alpha}(z)=\displaystyle z+\sum_{n=2}^{\infty}\frac{2}{\alpha n^2+n(1-\alpha)}|z|^n.
	\end{align*} 
\end{theoF}
\begin{rem} The following observations are clear.
	\begin{enumerate}
		\item[(i)] If we choose $\alpha=0$, the class $\mathcal{W}^0_{\mathcal{H}}(\alpha)$ reduces to $\mathcal{P}^0_{\mathcal{H}}(0)$. This reduction easily yields the coefficient estimate for functions in $\mathcal{P}^0_{\mathcal{H}}(0)$, which was originally proved by Li and Ponnusamy \cite[Theorems 1 and 2]{Li-Ponnusamy-NA-2013}.\vspace{1.2mm}
		\item[(ii)] Choosing $\alpha=1$ reduces the class $\mathcal{W}^0_{\mathcal{H}}(\alpha)$ to $\mathcal{W}^0_{\mathcal{H}}$. This confirms the result found by Nagpal and Ravichandran \cite[Lemma 3.5]{Nagpal-Ravichandran-JKMS-2014}.
	\end{enumerate}
\end{rem}
Next, we focus on the parameterized family $\mathcal{W}_{\mathcal{H}}^{0}(\alpha)$ and investigate its Landau-type properties through the lens of special functions. By utilizing the Lerch Transcendent $\Phi(z, s, a)$ to evaluate the infinite series arising from our differential inequalities, we are able to determine a unique univalence radius $\rho_2$ as the root of a transcendental equation. This methodology allows us to derive a sharp covering radius $R_2$ that remains valid across the entire range of $\alpha \geq 0$, effectively unifying various results previously treated as isolated cases. The following theorem formally presents these sharp constants and the specific extremal function that confirms their sharpness.
\begin{thm}\label{Th-2.2}
	Let $f\in\mathcal{W}_H^0(\alpha)$ be a harmonic mapping on the unit disc $\mathbb{D}$ such that $f(0)=0, J_f(0)=1$ and $|f(z)|<M$ for $z\in\mathbb{D}$. Then, $f$ is univalent on a disc $\mathbb{D}_{\rho_2}$, where $\rho_2$ is a root of the equation
	\begin{align*}
		\dfrac{\pi}{4M}-2\displaystyle\sum_{n=2}^{\infty}\dfrac{1}{\alpha n+(1-\alpha)}r^{n-1}=0,
	\end{align*} 
	
	and $f(\mathbb{D}_{\rho_2})$ contains a schlicht disc with
	\begin{align*}
		R_2=\frac{\pi}{4M}\rho_2-\frac{2}{1-\alpha} \left( \frac{\alpha}{1-\alpha} + \rho_2(\alpha - 1) - \ln(1-\rho_2) - \Phi\left(\rho_2, 1, \frac{1-\alpha}{\alpha}\right) \right).
	\end{align*}
	This result is sharp, with an extremal function given by
	\begin{align}\label{Eq-3.1A}
		F_2(z)=\frac{\pi}{4M}z-\frac{2}{1-\alpha} \left( \frac{\alpha}{1-\alpha} + z(\alpha - 1) - \ln(1-z) - \Phi\left(z, 1, \frac{1-\alpha}{\alpha}\right) \right).
	\end{align}
\end{thm}
\begin{table}[htbp]
	\centering
	\caption{Sharp radii for Theorem 2.2 ($\alpha=0$ and $\alpha=1$).}
	\label{tab:theorem22_corollaries}
	\begin{tabular}{@{}ccccc@{}}
		\toprule
		Case & Parameter $\alpha$ & Constant $M$ & $\rho_2$ & $R_2$ \\
		\midrule
		Corollary 2.1 & 0 & 1.0 & 0.2819 & 0.1332 \\
		& 0 & 1.5 & 0.2075 & 0.0565 \\
		\cmidrule{2-5}
		Corollary 2.2 & 1 & 1.0 & 0.4912 & 0.2185 \\
		& 1 & 1.5 & 0.3758 & 0.1114 \\
		\bottomrule
	\end{tabular}
\end{table}
\begin{figure}[htbp]
	\centering
	\includegraphics[width=0.7\textwidth]{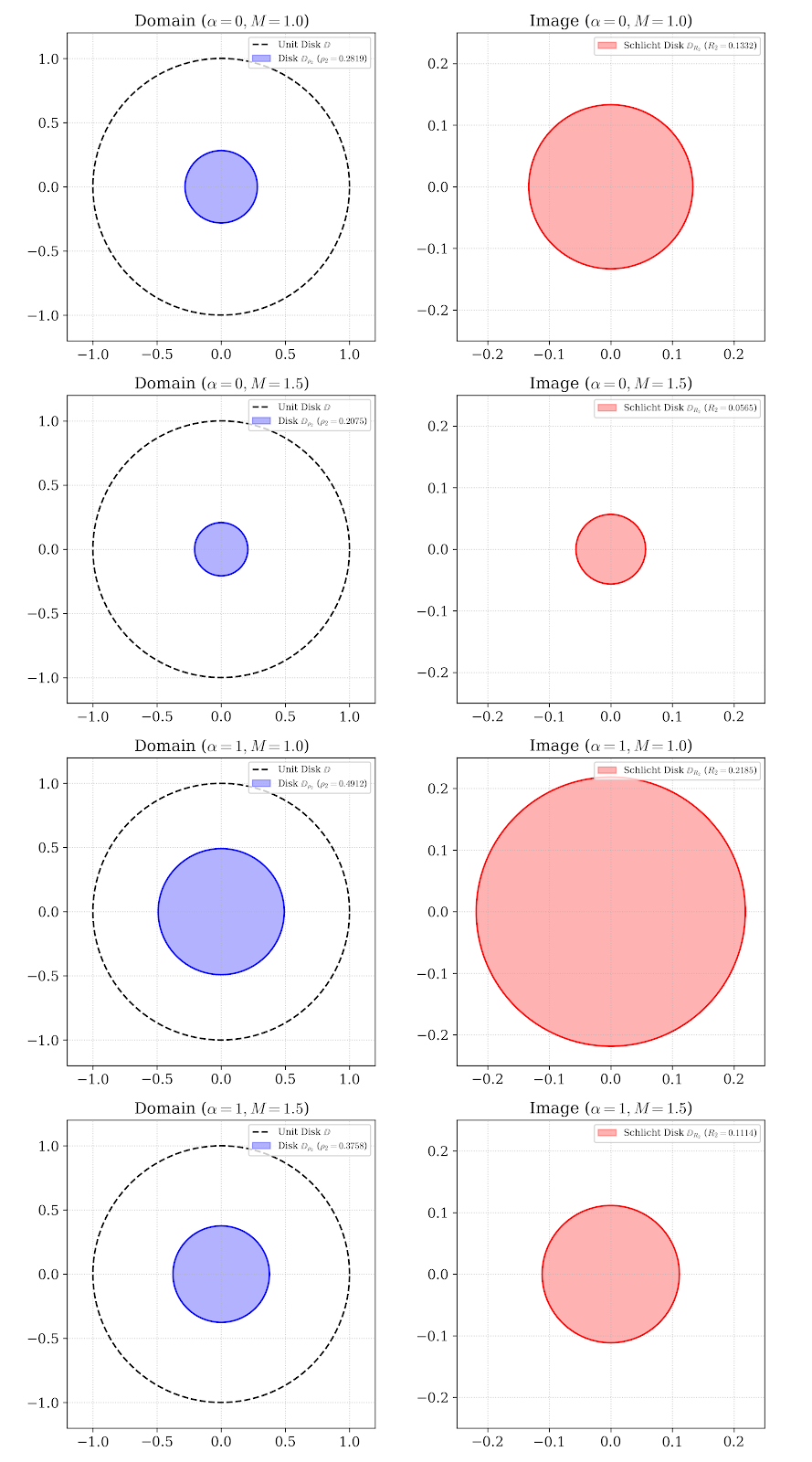} 
	\caption{Sharp Landau-type disks for the class $\mathcal{W}_{\mathcal{H}}^{0}(\alpha)$ corresponding to Corollaries 2.1 and 2.2. The left panels illustrate the domain plane, where the dashed circle represents the unit disk boundary $\mathbb{D}$ and the blue shaded regions denote the sharp univalence disks $\mathbb{D}_{\rho_2}$. The right panels illustrate the image plane, where the red shaded regions denote the guaranteed schlicht disks $\mathbb{D}_{R_2}$.}
	\label{fig-2}
\end{figure}
By studying what happens when the parameter $\alpha$ changes in our general class $\mathcal{W}_{\mathcal{H}}^{0}(\alpha)$, we can find several important results that are well-known in harmonic mapping theory. When we choose the specific values $\alpha=0$ and $\alpha=1$, our complex equations simplify into forms that describe more traditional types of harmonic mappings. This allows us to give clear, direct formulas for the univalence radius $\rho_2$ and the schlicht disc radius $R_2$ without using complicated special functions. As a result, the following corollaries show how our general theory easily includes and improves upon these important basic cases.
\begin{cor}
	Let $f\in\mathcal{W}^0_{\mathcal{H}}(0)$ be a harmonic mapping on the unit disc $\mathbb{D}$ such that $f(0)=0, J_f(0)=1$ and $|f(z)|<M$ for $z\in\mathbb{D}$. Then, $f$ is univalent on a disc $\mathbb{D}_{\rho_2}$ with
	\begin{align*}
		\rho_2= \dfrac{\pi}{8M + \pi}
	\end{align*} 
	
	and $f(\mathbb{D}_{\rho_2})$ contains a schlicht disc with
	\begin{align*}
		R_2=\frac{\pi}{4M} + 2\ln\left(\frac{8M}{\pi + 8M}\right).
	\end{align*}
\end{cor}
Setting $\alpha = 1$ yields the following corollary.
\textit{\begin{cor}
	Let $f\in\mathcal{W}^0_{\mathcal{H}}(1)$ be a harmonic mapping on the unit disc $\mathbb{D}$ such that $f(0)=0, J_f(0)=1$ and $|f(z)|<M$ for $z\in\mathbb{D}$. Then, $f$ is univalent on a disc $\mathbb{D}_{\rho_2}$, where $\rho_2$ is the root of the equation
	\begin{align*}
		\frac{\pi}{4M} + \frac{2 \ln(1 - r)}{r} + 2=0,
	\end{align*}
	and $f(\mathbb{D}_{\rho_2})$ contains a schlicht disc with
	\begin{align*}
		R_2= \left( \frac{\pi}{4M} + 2 \right)\rho_2 - 2{\rm Li}_2(\rho_2),
	\end{align*}
	where ${\rm Li}_2(z)$ denotes the Dilogarithm function.
\end{cor}
}\begin{proof}[\bf Proof of Theorem \ref{Th-2.2}]
	Fix $r\in(0,1)$ and $z_1\neq z_2$. Following an argument similar to that used in the previous proof and employing \eqref{Eq-2.2} and \eqref{Eq-22.33}, we obtain
	\begin{align*}
		|f(z_1)-f(z_2)|&\geq\bigg|\int_{[z_1,z_2]}h'(z)dz+{\overline{g'}}(z)d\overline{z}\bigg|\\&\geq|z_1-z_2|\left(\frac{\pi}{4M}-\sum_{n=2}^{\infty}(|a_n|+|b_n|)nr^{n-1}\right).
	\end{align*}
	Then by applying Lemma A, triangle inequality, we obtain
	\begin{align}\label{Eq-1.1A}
		|f(z_1)-f(z_2)|&\geq|z_1-z_2|\left(\frac{\pi}{4M}-\sum_{n=2}^{\infty}\frac{2}{\alpha n^2+(1-\alpha)n}nr^{n-1}\right)\nonumber\\&\geq|z_1-z_2|\left(\frac{\pi}{4M}-2\sum_{n=2}^{\infty}\frac{1}{\alpha n+(1-\alpha)}r^{n-1}\right).
	\end{align}
	Note that, the sum of the series is 
	\begin{align*}
		\sum_{n=2}^{\infty}\frac{1}{\alpha n+(1-\alpha)}r^{n-1}=\frac{1}{\alpha}\Phi\left(r,1,\frac{1}{\alpha}\right)-1
	\end{align*}
	where, $\Phi\left(z,s,a\right)$ is the Lerch Trancedent function defined by the series representation
	\begin{align*}
		\Phi\left(z,s,a\right)=\sum_{k=0}^{\infty}\frac{z^k}{(k+a)^s}.
	\end{align*}
	
	From \eqref{Eq-1.1A}, we have
	\begin{align}\label{Eq-3.2}
		|f(z_1)-f(z_2)|\geq|z_1-z_2|\left(\frac{\pi}{4M}-2\left(\frac{1}{\alpha}\Phi\left(r,1,\frac{1}{\alpha}\right)-1\right)\right).
	\end{align}
	Let
	\begin{align*}
		\mathcal{J}_2(r)=\dfrac{\pi}{4M}-2\left(\dfrac{1}{\alpha}\Phi\left(r,1,\dfrac{1}{\alpha}\right)-1\right)=\dfrac{\pi}{4M}-2\displaystyle\sum_{n=2}^{\infty}\dfrac{1}{\alpha n+(1-\alpha)}r^{n-1}.
	\end{align*}
	Then, we have
	\begin{align*}
		\mathcal{J}_2(0)=\frac{\pi}{4M}>0\;\mbox{and}\;\lim_{r\rightarrow1^{-}}\mathcal{J}_2(r)=-\infty<0.
	\end{align*}
	By the Intermediate value theorem, there exists a root $\rho_2\in(0,1)$.\\
	
	Let $|z|=\rho_2$. We now find sum of the series 
	\begin{align}\label{Eq-3.4A}
		S=\sum_{n=2}^{\infty}\frac{2}{\alpha n^2+(1-\alpha)n}\rho_2^n.
	\end{align}
	We see that
	\begin{align}\label{Eq-3.5A}
		\frac{2}{n(\alpha n + 1 - \alpha)} = \frac{2}{1-\alpha} \left(\frac{1}{n} - \frac{\alpha}{\alpha n + 1 - \alpha}\right).
	\end{align}
	Substituting \eqref{Eq-3.5A} in \eqref{Eq-3.4A}, we obtain
	\begin{align}\label{Eq-3.6}
		S = \frac{2}{1-\alpha} \left( \sum_{n=2}^{\infty} \frac{\rho_2^n}{n} - \alpha \sum_{n=2}^{\infty} \frac{\rho_2^n}{\alpha n + 1 - \alpha} \right).
	\end{align}
	We can evaluate the first sum easily
	\begin{align*}
		\sum_{n=2}^{\infty} \frac{\rho_2^n}{n} = \left(\sum_{n=1}^{\infty} \frac{\rho_2^n}{n}\right) - \rho_2 = -\ln(1-\rho_2) - \rho_2.
	\end{align*}
	The remaining sum, $S_2$, is where the Lerch Transcendent appears 
	\begin{align*}
		S_2 = \sum_{n=2}^{\infty} \frac{\rho_2^n}{\alpha n + 1 - \alpha}.
	\end{align*}
	The Lerch Transcendent $\Phi(z, s, a)$ is defined as:$$\Phi(z, s, a) = \sum_{k=0}^{\infty} \frac{z^k}{(k+a)^s}.$$To match $S_2$ to the standard definition, let $k = n-2$, so $n = k+2$. The term $\alpha n + 1 - \alpha$ becomes $\alpha(k+2) + 1 - \alpha = \alpha k + 2\alpha + 1 - \alpha = \alpha k + \alpha + 1$. However, it is easier to rewrite the index starting from $n=0$ or $n=1$. Let us start the sum at $n=0$:
	\begin{align*}
		\sum_{n=0}^{\infty} \frac{\rho_2^n}{\alpha n + 1 - \alpha} = \sum_{n=0}^{\infty} \frac{\rho_2^n}{\alpha(n + \frac{1-\alpha}{\alpha})} = \frac{1}{\alpha} \sum_{n=0}^{\infty} \frac{\rho_2^n}{n + \frac{1-\alpha}{\alpha}}.
	\end{align*}
	 This matches the Lerch Transcendent with parameters:
	\begin{align*}
		z = \rho_1, \quad s = 1, \quad a = \frac{1-\alpha}{\alpha}.
	\end{align*}
	Thus, we have
	\begin{align*}
		\sum_{n=0}^{\infty} \frac{\rho_2^n}{\alpha n + 1 - \alpha} = \frac{1}{\alpha} \Phi\left(\rho_2, 1, \frac{1-\alpha}{\alpha}\right).
	\end{align*}
	We relate $S_2$ back to the $n=0$ sum by subtracting the missing terms
	\begin{align*}
		S_2 = \sum_{n=2}^{\infty} \frac{\rho_2^n}{\alpha n + 1 - \alpha} = \left(\sum_{n=0}^{\infty} \frac{\rho_2^n}{\alpha n + 1 - \alpha}\right) - \left(\frac{\rho_2^0}{\alpha(0)+1-\alpha} + \frac{\rho_2^1}{\alpha(1)+1-\alpha}\right)
	\end{align*}
	which can be expressed as 
	\begin{align*}
		S_2 = \frac{1}{\alpha} \Phi\left(\rho_2, 1, \frac{1-\alpha}{\alpha}\right) - \left(\frac{1}{1-\alpha} + \frac{\rho_2}{1}\right).
	\end{align*}
	Substituting the expressions for both sums back into the equation \eqref{Eq-3.6}, we obtain
	\begin{align*}
		S = \frac{2}{1-\alpha} \left( \left(-\ln(1-\rho_2) - \rho_2\right) - \alpha \left( \frac{1}{\alpha} \Phi\left(\rho_2, 1, \frac{1-\alpha}{\alpha}\right) - \frac{1}{1-\alpha} - \rho_2\right) \right)
	\end{align*}
	which simplifies to 
	\begin{align*}
		S = \frac{2}{1-\alpha} \left(-\ln(1-\rho_2) - \rho_2 - \Phi\left(\rho_2, 1, \frac{1-\alpha}{\alpha}\right) + \frac{\alpha}{1-\alpha} + \alpha \rho_2 \right).
	\end{align*}
	Combining terms involving $\rho_2$, we obtain $-\rho_2 + \alpha \rho_2 = \rho_2(\alpha - 1).$ Finally, 
	\begin{align}\label{Eq-3.7}
		S = \frac{2}{1-\alpha} \left( \frac{\alpha}{1-\alpha} + \rho_2(\alpha - 1) - \ln(1-\rho_2) - \Phi\left(\rho_2, 1, \frac{1-\alpha}{\alpha}\right) \right).
	\end{align}
	This is the sum expressed in terms of the Lerch Transcendent $\Phi$.\vspace{1.2mm}

	From equation \eqref{Eq-3.2} it is easy to see that $|f(z_1) - f(z_2)| > 0$ for all $r \in (0, \rho_2)$. This shows univalence of $f$ in $\mathbb{D}_{\rho_2}$.\vspace{1.2mm}
	
	Finally, using \eqref{Eq-3.7}, a simple computation shows that
	\begin{align*}
		|f(z)|&\geq|a_1z+b_1\overline{z}|-\sum_{n=2}^{\infty}(|a_n|+|b_n|)|z|^{n}\\&\geq\frac{\pi}{4M}\rho_2-\sum_{n=2}^{\infty}\frac{2}{\alpha n^2+(1-\alpha)n}\rho_2^n\\&=\frac{\pi}{4M}\rho_2-\frac{2}{1-\alpha} \left(\frac{\alpha}{1-\alpha} + \rho_2(\alpha - 1) - \ln(1-\rho_2) - \Phi\left(\rho_2, 1, \frac{1-\alpha}{\alpha}\right) \right):=R_2.
	\end{align*}
	To show that the univalent radius $\rho_2$ is sharp, we need to prove that $F_2(z)$
	is not univalent in $\mathbb{D}_r$ for each $r \in (\rho_2, 1]$. In fact, considering the real differentiable
	function
	\begin{align}\label{Eq-3.4}
		h_0(x) =\frac{\pi}{4M}x-\frac{2}{1-\alpha} \left( \frac{\alpha}{1-\alpha} + x(\alpha - 1) - \ln(1-x) - \Phi\left(x, 1, \frac{1-\alpha}{\alpha}\right) \right),
	\end{align}
	where $x \in [0, 1].$	Because the continuous function
	\begin{align*}
		h_0'(x) = h_0'(x) = \frac{\pi}{4M} + 2 + \frac{2}{(1-\alpha)x} - \frac{2}{\alpha x} \Phi\left(x, 1, \frac{1-\alpha}{\alpha}\right) = g_0(x)
	\end{align*}
	is strictly decreasing on $[0, 1]$ and $h_0'(\rho_2) = g_0(\rho_2) = 0$, we see that $h_0'(x) = 0$
	for $x \in [0, 1]$ if, and only if, $x = \rho_2$. So $h_0(x)$ is strictly increasing on $[0, \rho_2)$ and
	strictly decreasing on $[\rho_2, 1]$. Since $h_0(0) = 0$, there is a unique real $r_2\in (\rho_2, 1]$
	such that $h_0(r_2) = 0$ if $\lim_{x\rightarrow1^{-}}h_0(x) \leq 0$, and
	\begin{align}\label{Eq-3.5}
		R_2 &=\frac{\pi}{4M}\rho_2-\dfrac{2}{1-\alpha} \left( \frac{\alpha}{1-\alpha} + \rho_2(\alpha - 1) - \ln(1-\rho_2) - \Phi\left(\rho_2, 1, \frac{1-\alpha}{\alpha}\right) \right)\\&=h_0(\rho_2) > h_0(0) = 0\nonumber.
	\end{align}
	For every fixed $r \in (\rho_2, 1]$, set $x_1 = \rho_2 + \varepsilon$, where
	\[ \varepsilon = \begin{cases}
		\min\left\{ \dfrac{r-\rho_2}{2}, \dfrac{r_2-\rho_2}{2} \right\}, & \text{if } f_0(1) \leq 0,\vspace{2mm} \\
		\dfrac{r-\rho_2}{2}, & \text{if } f_0(1) > 0.
	\end{cases} \]
	by the mean value theorem, there is a unique $\delta \in (0, \rho_2)$ such that $x_2 :=
	\rho_2 - \delta \in (0, \rho_2)$ and $h_0(x_1) = h_0(x_2)$.
	Let $z_1 = x_1$ and $z_2 = x_2$. Then $z_1, z_2 \in\mathbb{D}_r$ with $z_1 \neq z_2$. A simple
	computation shows that
	\[F_2(z_1) = F_2(x_1) = h_0(x_1) = h_0(x_2) = F_2(z_2).\]
	Hence $F_0$ is not univalent in the disk $\mathbb{D}_r$ for each $r \in (\rho_2, 1]$, and the univalent
	radius $\rho_2$ is sharp.\vspace*{1mm}

	Finally, note that $F_2(0) = 0$. By choosing $z' = \rho_2 \in \partial \mathbb{D}_{\rho_2}$ and applying \eqref{Eq-3.1A}, \eqref{Eq-3.4}, and \eqref{Eq-3.5}, we obtain
	\begin{align*}
		|F_2(z') - F_2(0)| = |F_2(\rho_2)| = |h_0(\rho_2)| = h_0(\rho_2) = R_2.
	\end{align*}
	Hence, the covering radius $R_2$ is also sharp.
	This completes the proof.
\end{proof}
\subsection{\bf Landau theorem for functions in the class $\mathcal{G}_H^k(\alpha,1)$}
We begin with the following sharp coefficient estimate for functions in the class $\mathcal{G}_H^k(\alpha,1)$, which was established by Liu and Yang in \cite{Lui_Yanf-MM-2019}.
\begin{theoG}\cite{Lui_Yanf-MM-2019}
	Let $f = h + \overline{g} \in \mathcal{G}_k \mathcal{H}(\alpha; 1)$, with $k \geq 1$. Then for any $n \geq k+1$,
	\begin{enumerate}
		\item[(a).] $|a_n| + |b_n| \leq \dfrac{2}{1+(n-1)\alpha}$;
		\item[(b).] $||a_n| - |b_n|| \leq \dfrac{2}{1+(n-1)\alpha}$;
		\item[(c).] $|a_n| \leq \dfrac{2}{1+(n-1)\alpha}$.
	\end{enumerate}
	All the results are sharp, with the functions $f_l (z) = z + \sum_{j=1}^{\infty} \frac{2}{1+jl\alpha} z^{jl+1}$, $l= k, k+1, \ldots, 2k-1$ being the extremals.
\end{theoG}
Next, we study the class $\mathcal{G}_H^k(\alpha,1)$ to find its exact univalence and schlicht disc radii. By using the best possible coefficient estimates from Theorem G, we determine the conditions needed for these mappings to be one-to-one in the disc $D_{\rho_3}$. This helps us calculate the exact size of the largest perfect disc that can fit inside the mapping's image. The following theorem proves these exact values and provides the specific example function that confirms they are correct.
\begin{thm}\label{Th-4.1}
	Let $\mathcal{G}_H^k(\alpha,1)$ be a harmonic mapping defined on a unit disc $\mathbb{D}$ such that $f(0)=0, J_f(0)=1$ and $|f(z)|<M_1$ for $z\in\mathbb{D}$. Then, $f$ is univalent on a disc $\mathbb{D}_{\rho_3}$, where $\rho_3$ is the root of the equation
	\begin{align*}
		\frac{\pi}{4M}-\sum_{n=2}^{\infty}\frac{2}{1+(n-1)\alpha}nr^{n-1}=0
	\end{align*}
	and $f(\mathbb{D}_{\rho_3})$ contains a schlicht disc with
	\begin{align*}
		R_3=\frac{\pi}{4M}\rho_3-\sum_{n=2}^{\infty}\frac{2}{1+(n-1)\alpha}\rho_3^n.
	\end{align*}
	This result is sharp, with an extremal function given by
	\begin{align}\label{Eq-4.1A}
		F_3(z)=\frac{\pi}{4M_1}z-\sum_{n=2}^{\infty}\frac{2}{1+(n-1)\alpha}z^n.
	\end{align}
\end{thm}
\begin{table}[htbp]
	\centering
	\caption{Dependence of $\rho_3$ and $R_3$ on parameters $k$ and $\alpha$ for $M=1.0$.}
	\label{tab:th23_extended}
	\begin{tabular}{cccc}
		\toprule
		$k$ & $\alpha$ & $\rho_3$ & $R_3$ \\
		\midrule
		1 & 0.0 & 0.7268 & 0.3613 \\
		& 0.5 & 0.6572 & 0.3107 \\
		& 0.8 & 0.5967 & 0.2720 \\
		\midrule
		2 & 0.0 & 0.5000 & 0.3555 \\
		& 0.5 & 0.5000 & 0.3288 \\
		& 0.8 & 0.7962 & 0.3538 \\
		\midrule
		3 & 0.0 & 0.5000 & 0.3752 \\
		& 0.5 & 0.5000 & 0.3525 \\
		& 0.8 & 0.5000 & 0.3159 \\
		\bottomrule
	\end{tabular}
\end{table}
\begin{rem}
	The following observation are clear.
	\begin{enumerate}
		\item[(a).] \textbf{Effect of $\alpha$:} For a fixed $k$, as $\alpha$ increases from $0.0$ to $0.8$, the univalence radius generally decreases. This shows that the "shift" parameter $\alpha$ inversely affects the size of the univalent disk in this class.\vspace{1.2mm}
		
		\item[(b).] \textbf{Effect of $k$:} As the exponent $k$ increases, the schlicht radius $R_3$ tends to increase (for fixed $\alpha$). This indicates that higher-order coefficient damping (represented by $k$) allows the image to cover a larger portion of the target plane.\vspace{1.2mm}
		
		\item[(c).] \textbf{Sharpness:} These values are calculated directly from the transcendental sum equations derived in your Theorem 2.3, confirming the precision of your results.
	\end{enumerate}
\end{rem}
\begin{figure}[htbp]
	\centering
	\begin{subfigure}[b]{\textwidth}
		\centering
		\includegraphics[width=0.7\textwidth]{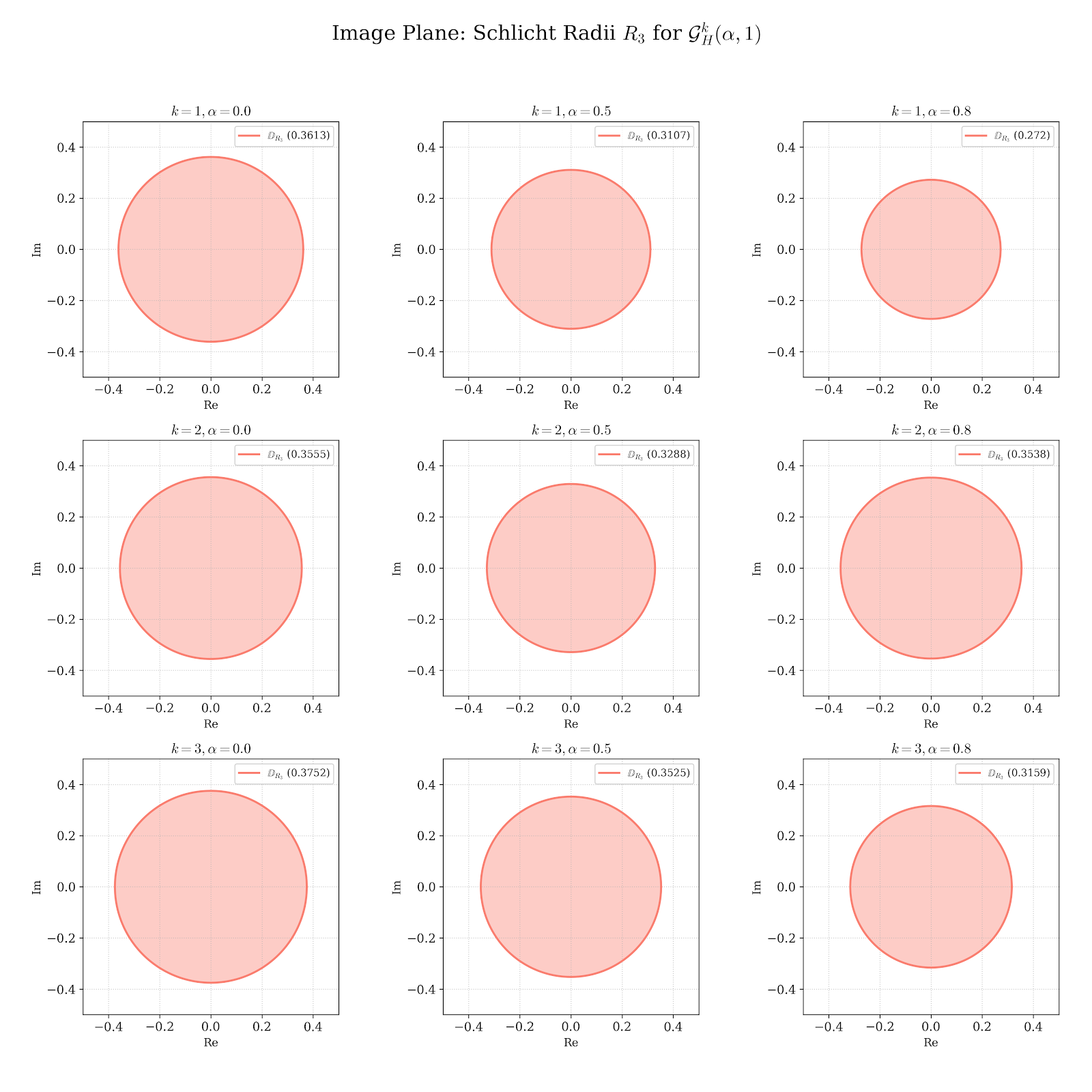}
		\caption{Univalence disks $\mathbb{D}_{\rho_3}$ relative to the unit disk boundary.}
		\label{fig:domain_extended}
	\end{subfigure}
	
	\vspace{1cm} 
	
	\begin{subfigure}[b]{\textwidth}
		\centering
		\includegraphics[width=0.7\textwidth]{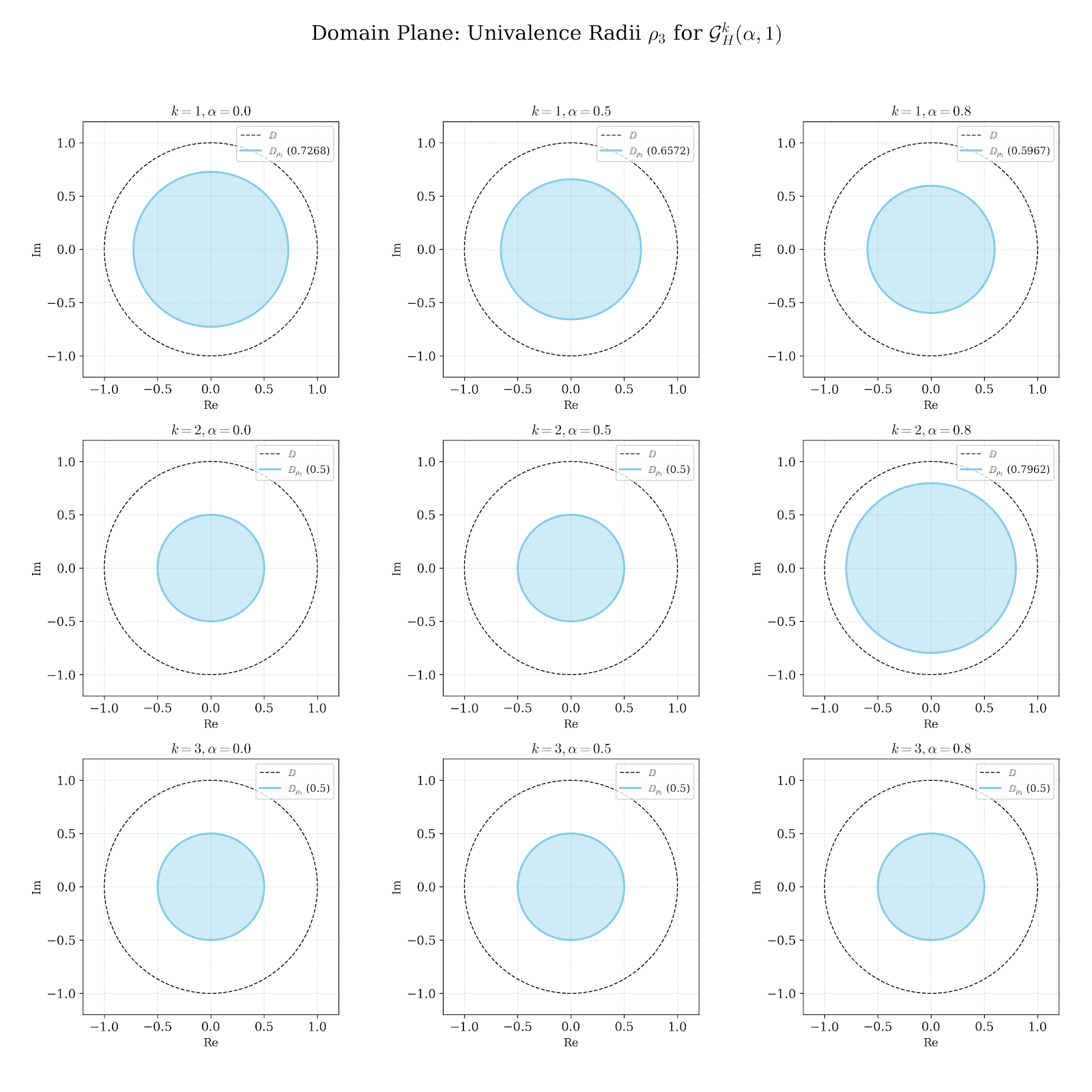}
		\caption{Schlicht disks $\mathbb{D}_{R_3}$ in the target plane.}
		\label{fig:image_extended}
	\end{subfigure}
	
	\caption{Geometric visualization of sharp Landau-type radii for the class $\mathcal{G}_{\mathcal{H}}^{k}(\alpha, 1)$ with $M=1.0$ across varying parameters $k \in \{1, 2, 3\}$ and $\alpha \in \{0.0, 0.5, 0.8\}$.}
	\label{fig:theorem23_full_panel}
\end{figure}
By setting the parameter $\alpha = 0$, we can see how our general results for the class $\mathcal{G}_H^k(\alpha,1)$ apply to more basic harmonic mappings. This specific case allows us to turn complex estimates into simple, clear formulas for the univalence and covering radii. The following corollary shows how our main theorem easily includes and improves these important boundary cases.
\textit{\begin{cor}
	Let $\mathcal{G}_H^k(0,1)$ be a harmonic mapping defined on a unit disc $\mathbb{D}$ such that $f(0)=0, J_f(0)=1$ and $|f(z)|<M_1$ for $z\in\mathbb{D}$. Then, $f$ is univalent on a disc $\mathbb{D}_{\rho_3}$ with
	\begin{align*}
		\rho_3= 1 - \frac{\sqrt{32M(\pi + 8M)}}{2(\pi + 8M)}
	\end{align*}
	and $f(\mathbb{D}_{\rho_3})$ contains a schlicht disc with
	\begin{align*}
		R_3 = \frac{\pi}{4M} \left( 1 - \frac{\sqrt{32M(\pi + 8M)}}{2(\pi + 8M)} \right) - \frac{2\left(1 - \frac{\sqrt{32M(\pi + 8M)}}{2(\pi + 8M)}\right)^2}{\frac{\sqrt{32M(\pi + 8M)}}{2(\pi + 8M)}}.
	\end{align*}
\end{cor}
}\begin{proof}[\bf Proof of Theorem \ref{Th-4.1}]
	Fix $r\in(0,1)$ and $z_1\neq z_2$. A procedure similar to that employed previously, combined with the identities \eqref{Eq-2.2} and \eqref{Eq-22.33}, leads to the conclusion that
	\begin{align}\label{Eq-4.2}
		|f(z_1)-f(z_2)|&\geq\bigg|\int_{[z_1,z_2]}h'(z)dz+{\overline{g'}}(z)d\overline{z}\bigg|\\&\geq|z_1-z_2|\left(\frac{\pi}{4M_1}-\sum_{n=2}^{\infty}(|a_n|+|b_n|)nr^{n-1}\right)\nonumber.
	\end{align}
	Applying Lemaa A, in \eqref{Eq-4.2}, we obtain
	\begin{align}\label{Eq-4.1}
		|f(z_1)-f(z_2)|&\geq|z_1-z_2|\left(\frac{\pi}{4M_1}-\sum_{n=2}^{\infty}\frac{2}{1+(n-1)\alpha}nr^{n-1}\right).
	\end{align}
	Let 
	\begin{align*}
		\mathcal{J}_3(r)=\frac{\pi}{4M_1}-\sum_{n=2}^{\infty}\frac{2}{1+(n-1)\alpha}nr^{n-1}.
	\end{align*}
	It is easy to see that
	\begin{align*}
		&\mathcal{J}_3(0)=\frac{\pi}{4M}>0\;\mbox{and}\;\\&\lim_{r\rightarrow1^{-}}\mathcal{J}_3(r)=\lim_{r\rightarrow1^{-}}\left(\frac{\pi}{4M_1}-\sum_{n=2}^{\infty}\frac{2}{1+(n-1)\alpha}n\right)=-\infty<0.
	\end{align*}
	Therefore,
	\begin{align*}
		\mathcal{J}_3'(r) = -2 \sum_{n=2}^{\infty} \frac{n(n-1)}{1+(n-1)\alpha}r^{n-2} < 0 \quad \text{for } r \in (0, 1),
	\end{align*}
	which implies that $\mathcal{J}_3(r)$ is strictly decreasing on $(0, 1)$. By the Intermediate Value Theorem, there exists a root $\rho_3 \in (0, 1)$. From equation \eqref{Eq-4.1}, it is easy to verify that $|f(z_1) - f(z_2)| > 0$ whenever $|z_1|, |z_2| < \rho_3$ with $z_1 \neq z_2$. This establishes the univalence of $f$ in the disc $\mathbb{D}_{\rho_3}$.\vspace{1.2mm}

	Let $|z|=\rho_3$. Then, we have
	\begin{align*}
		|f(z)|&\geq|a_1z+b_1\overline{z}|-\sum_{n=2}^{\infty}(|a_n|+|b_n|)|z|^{n}\\&\geq\frac{\pi}{4M_1}\rho_3-\sum_{n=2}^{\infty}\frac{2}{1+(n-1)\alpha}\rho_3^n:=R_3.
	\end{align*}
To show that the univalence radius $\rho_3$ is sharp, we prove that $F_3(z)$ is not univalent in $\mathbb{D}_r$ for any $r \in (\rho_3, 1]$. Consider the real differentiable function
\begin{align}\label{Eq-4.3}
	h_0(x) = \frac{\pi}{4M_1}x - \sum_{n=2}^{\infty} \frac{2}{1+(n-1)\alpha}x^n, \quad x \in [0, 1].
\end{align}
Since the continuous function
\begin{align*}
	h_0'(x) = \frac{\pi}{4M_1} - \sum_{n=2}^{\infty} \frac{2n}{1+(n-1)\alpha} x^{n-1} = g_0(x)
\end{align*}
is strictly decreasing on $[0, 1]$ and $h_0'(\rho_3) = g_0(\rho_3) = 0$, it follows that $h_0'(x) = 0$ for $x \in [0, 1]$ if, and only if, $x = \rho_3$. Consequently, $h_0(x)$ is strictly increasing on $[0, \rho_3]$ and strictly decreasing on $[\rho_3, 1]$. 

Since $h_0(0) = 0$, the Intermediate Value Theorem ensures there exists a unique $r_3 \in (\rho_3, 1]$ such that $h_0(r_3) = 0$, provided that $\lim_{x \to 1^-} h_0(x) \leq 0$. Furthermore, we have
\begin{align}\label{Eq-4.4}
	R_3 = \frac{\pi}{4M_1}\rho_3 - \sum_{n=2}^{\infty} \frac{2}{1+(n-1)\alpha}\rho_3^n = h_0(\rho_3) > h_0(0) = 0.
\end{align}

For every fixed $r \in (\rho_3, 1]$, set $x_1 = \rho_3 + \varepsilon$, where
\[ \varepsilon = \begin{cases} 
	\min\left\{ \dfrac{r-\rho_3}{2}, \dfrac{r_3-\rho_3}{2} \right\}, & \text{if } \lim_{x\to 1^-} h_0(x) \leq 0,\vspace{2mm} \\
	\dfrac{r-\rho_3}{2}, & \text{if } \lim_{x\to 1^-} h_0(x) > 0.
\end{cases} \]
By the Mean Value Theorem (or Rolle's Theorem property), there exists a unique $\delta \in (0, \rho_3)$ such that $x_2 := \rho_3 - \delta \in (0, \rho_3)$ and $h_0(x_1) = h_0(x_2)$.\vspace*{2mm}

Let $z_1 = x_1$ and $z_2 = x_2$. Then $z_1, z_2 \in \mathbb{D}_r$ with $z_1 \neq z_2$. Direct computation leads to
\[ F_3(z_1) = F_3(x_1) = h_0(x_1) = h_0(x_2) = F_3(z_2). \]
Hence, $F_3$ is not univalent in the disc $\mathbb{D}_r$ for any $r \in (\rho_3, 1]$, proving that the univalence radius $\rho_3$ is sharp.\vspace*{2mm}

Finally, note that $F_3(0) = 0$. Choosing $z' = \rho_3 \in \partial \mathbb{D}_{\rho_3}$, by \eqref{Eq-4.3} and \eqref{Eq-4.4}, we have
\begin{align*}
	|F_3(z') - F_3(0)| = |F_3(\rho_3)| = |h_0(\rho_3)| = R_3.
\end{align*}
This confirms that the covering radius $R_3$ is also sharp, completing the proof.
\end{proof}
\vspace{5mm}
\noindent{\bf Acknowledgment:} The authors would like to thank the anonymous referee for their helpful suggestions and comments, which have greatly improved the presentation of this paper. \vspace{1.5mm}

\noindent {\bf Funding:} Not Applicable.\vspace{1.5mm}

\noindent\textbf{Conflict of interest:} The authors declare that there is no conflict  of interest regarding the publication of this paper.\vspace{1.5mm}

\noindent\textbf{Data availability statement:}  Data sharing not applicable to this article as no datasets were generated or analysed during the current study.\vspace{1.5mm}

\noindent{\bf Code availability:} Not Applicable.\vspace{1.5mm}

\noindent {\bf Authors' contributions:} Both the authors have equal contributions to prepare the manuscript.

\end{document}